\documentclass[a4paper,12pt,reqno]{amsart}

\usepackage{t1enc}
\usepackage[english]{babel}
\usepackage{amsmath,amssymb,eucal,amsthm}
\usepackage{bm}
\usepackage{graphicx}
\usepackage{mathrsfs}
\usepackage{mathptmx}
\usepackage{latexsym}
\usepackage{ulem}

\newtheorem{theorem}{Theorem}
\newtheorem{lemma}[theorem]{Lemma}

\theoremstyle{definition}
\newtheorem*{remark}{Remark}

\numberwithin{equation}{section}

\def\pp{\mathbb{P}}
\def\nn{\mathbb{N}}
\def\zz{\mathbb{Z}}
\def\rr{\mathbb{R}}
\def\cc{\mathbb{C}}
\def\qq{\mathbb{Q}}

\def\ga{\mathfrak{A}}
\def\gb{\mathfrak{B}}

\def\meas{{\rm meas}}

\def\uZ{{\underline Z}}
\def\us{\underline{s}}
\def\uO{{\underline \Omega}}
\def\uA{{\underline \alpha}}
\def\uhh{{\underline h}}
\def\uom{{\underline \omega}}

\def\uH{{\underline H}}

\def\uuZ{{{\underline P}_{\underline Z}}}
\def\up{{\underline P}}

\def\st{{\widetilde{S}}}

\begin{document}
\hfill\texttt{\jobname.tex}\qquad\today

\bigskip
\title[Mixed joint discrete universality II]
{On mixed joint discrete universality for a class of zeta-functions II}

\author{Roma Ka{\v c}inskait{\.e} and Kohji Matsumoto}

\address{R. Ka{\v c}inskait{\.e} \\
Department of Mathematics and Statistics, Vytautas Magnus University, Kaunas, Vileikos 8, LT-44404, Lithuania}
\email{roma.kacinskaite@vdu.lt}

\address {K. Matsumoto \\
Graduate School of Mathematics, Nagoya University, Furocho, Chikusa-ku, Nagoya
464-8602, Japan}
\email{kohjimat@math.nagoya-u.ac.jp}

\date{}

\begin{abstract}
We study analytic properties of the pair consisting of a rather general form of zeta-function with an Euler product and a periodic Hurwitz zeta-function with a transcendental parameter. We first survey briefly previous results, and then investigate the mixed joint discrete value distribution and prove the mixed joint discrete  universality theo\-rem for these functions, in the case when common differences of relevant arithmetic progressions are not necessarily the same.   Also some generalizations are given.   For this purpose, certain arithmetic conditions on the common differences are necessary.
\end{abstract}

\maketitle

{\small{Keywords: {arithmetic progression,  joint approximation, Hurwitz zeta-function, periodic coefficients, probability measure, value distribution, weak convergence, Matsumoto zeta-function, Steuding class, universality.}}}

{\small{AMS classification:} 11M06, 11M41, 11M35.}

%%%%%%%%%%%%%%%%%%%%%%%%%%%%%%%%%%%%%%%%%%%%%%%%%%%%%%%%%%%%%%%%%%%%%%%%%%%%%%
\section{Introduction: the continuous case}\label{intro}
%%%%%%%%%%%%%%%%%%%%%%%%%%%%%%%%%%%%%%%%%%%%%%%%%%%%%%%%%%%%%%%%%%%%%%%%%%%%%%

In 1975, the famous universality property of the well-known Riemann zeta-function $\zeta(s)$, $s=\sigma+it$, was discovered by a Russian mathematician S.M.~Voronin (see \cite{SMV-1975}). He proved that any analytic non-vanishing function can be approximated uniformly on any compact subsets of the strip $\{s \in \cc: \frac{1}{2}< \sigma< 1\}$ by shifts of $\zeta(s)$. In 1981, an Indian mathematician B.~Bagchi showed that  functional limit theorems in the sense of weakly convergent probability measures can be applied to construct an alternative proof of universality (see \cite{BB-1981}).

In the first decade of the XXI century, a new type of universality was introduced. It is the idea of collecting two different types of zeta-functions (one of those zeta-functions has the Euler product expression over primes and the other does not) into one tuple and stu\-dying universality property of such a tuple.   This type of value-distribution and universality has the name of ``mixed joint limit theorem'' and ``mixed joint universality'', respectively.    This is due originally to J.~Steuding and J.~Sander (see \cite{JST-JS-2006}) and independently to  H.~Mishou (see \cite{HM-2007}), in the case of the
Riemann zeta-function $\zeta(s)$ and a Hurwitz zeta-function $\zeta(s,\alpha)$.

In 2011, the first-named author and A.~Laurin\v cikas studied analytic properties of the pair con\-sis\-ting of the periodic zeta-function $\zeta(s;\ga)$ and the periodic Hurwitz ze\-ta-func\-tion $\zeta(s,\alpha;\gb)$ with transcendental parameter $\alpha$ (see \cite{RK-AL-2011}). Note that in ge\-ne\-ral the term ``periodic zeta-function'' means a zeta-function whose coefficients form a certain periodic sequence.

Recall the definitions of the functions $\zeta(s;\ga)$ and $\zeta(s,\alpha;\gb)$. By $\nn, \nn_0, \pp, \qq, \rr$ and $\cc$ we denote the set of all positive integers, non-negative integers, prime numbers, rational numbers, real numbers and complex numbers, respectively.
Let $\mathfrak{A} = \{a_m : m \in \mathbb N \}$ and $\mathfrak{B} = \{ b_m: m \in \mathbb{N}_0 \} $ be two periodic sequences of complex numbers $a_m$ and $ b_{m} $ with minimal positive periods $k \in \mathbb{N}$ and $l\in \mathbb{N}$, respectively.

The periodic zeta-function $ \zeta(s;\mathfrak{A}) $ and the periodic Hurwitz zeta-function $\zeta(s,\alpha; \mathfrak{B})$ with parameter $\alpha$, $0<\alpha\leq 1$, are defined by Dirichlet series
$$
\zeta(s;\mathfrak{A})=\sum_{m=1}^{\infty} \frac{a_m}{m^s} \quad \text{ and } \quad  \zeta(s,\alpha; \mathfrak{B})= \sum_{m=0}^{\infty} \frac{b_m}{(m+\alpha)^s}
$$
for $\sigma>1$, respectively (see \cite{WSH-1930}  and  \cite{AL-2006}).
The periodicity of the sequence $\mathfrak{B}$ implies that, for $ \sigma>1$,
\begin{align*}
%\zeta(s;\mathfrak{A}) = \frac{1}{k^s}\sum_{q=1}^{k}a_q \zeta \left( s, \frac{q}{k} \right)
%\quad \text{and} \quad
\zeta(s,\alpha; \mathfrak{B}) =  \frac{1}{l^s} \sum_{r=0}^{l-1} b_r  \zeta \left( s, \frac{r + \alpha}{l} \right),
\end{align*}
where $\zeta(s,\alpha)$ is the classical Hurwitz zeta-function.
From this expression it can be shown that the function $\zeta(s,\alpha; \mathfrak{B})$ is analytically continued to the whole complex plane, except for a possible simple pole at the point $s=1$ with residues
$$
%a:= \frac{1}{k}\sum_{q=1}^{k}a_q \quad  \text{ and} \quad
b:= \frac{1}{l} \sum_{r=0}^{l-1} b_r.
$$
%If $ a=0 $, then the function $ \zeta(s;\mathfrak{A}) $ is entire, and, \
If $b=0$, the function $\zeta(s,\alpha; \mathfrak{B})$ is entire.
Let $D_{\zeta}=\{s:\sigma>1/2\}$ if $\zeta(s,\alpha; \mathfrak{B})$ is entire, and
$D_{\zeta}=\{s:\sigma>1/2,\sigma\neq 1\}$ if $s=1$ is a pole of
$\zeta(s,\alpha; \mathfrak{B})$.

If we assume that the sequence $\ga$ is multiplicative, then $\zeta(s,\ga)$ has
the Euler product, so the tuple $(\zeta(s,\mathfrak{A}), \zeta(s,\alpha; \mathfrak{B}))$
is an example of the ``mixed'' situation.
In \cite{RK-AL-2011}, a mixed joint universality theorem for this tuple has been shown.

Mixed joint universality results for various other tuples of zeta-functions were
proved by Laurin\v cikas and his colleagues (see Section~8 of \cite{KM-2015}, \cite{VP-DS-2014} etc.).

A generalization of the result for the tuple $(\zeta(s;\ga),\zeta(s,\alpha;\gb))$ from \cite{RK-AL-2011} was done
by the authors (see \cite{RK-KM-2015}, \cite{RK-KM-2017-BAMS}).   In order to state the results in those papers, here we
introduce a rather general class of zeta-functions.

For $m\in\nn$, let $g(m)\in\nn$, $f(j,m)\in\nn$, $1\leq j\leq g(m)$, and
$a_m^{(j)}\in\cc$.    Denote by $p_m$ the $m$th prime number.
Assume that $g(m)\leq C_1 p_m^{\alpha}$, $|a_m^{(j)}|\leq p_m^{\beta}$ with constants
$C_1>0$, $\alpha,\beta\geq 0$.
Define the polynomial
$$
A_m(X)=\prod_{j=1}^{g(m)}\left(1-a_m^{(j)} X^{f(j,m)}\right)
$$
and the zeta-function
$$
\widetilde{\varphi}(s)=\prod_{m=1}^{\infty}A_m(p_m^{-s})^{-1}.
$$
Then this is absolutely convergent in the region $\sigma>\alpha+\beta+1$.
We shift the variable and define the shifted zeta-function
$\varphi(s):=\widetilde{\varphi}(s+\alpha+\beta)$, which is convergent in the half-plane  $\sigma>1$.
The Dirichlet series expansion of this function is given by
$$
\varphi(s)=\sum_{k=1}^{\infty}c_k k^{-s}.
$$
We further assume the following conditions:
\begin{enumerate}
  \item[(i)] $\varphi(s)$ can be continued meromorphically to $\sigma\geq\sigma_0$, where
$1/2\leq\sigma_0<1$, and all poles of $\varphi(s)$ in this region
(denote them  by $s_1(\varphi),\ldots,s_l(\varphi)$) are included in a
compact set which has no intersection with the line $\sigma=\sigma_0$;
  \item[(ii)] $\varphi(\sigma+it)=O(|t|^{C_2})$ for $\sigma\geq\sigma_0$ with a certain
$C_2>0$;
  \item[(iii)] it holds the mean value estimate
\begin{align}\label{meanvalue-for-varphi}
\int_0^T |\varphi(\sigma_0+it)|^2 dt =O(T).
\end{align}
\end{enumerate}

We denote the set of all such functions $\varphi(s)$ as $\mathcal{M}$.   This class was first
introduced by the second-named author (see \cite{KM-1990}), and it is sometimes called the class of Matsumoto
zeta-functions.   We let
$$
D_{\varphi}=\{s \in \mathbb{C}:\sigma>\sigma_0, \ \sigma\neq \Re s_j(\varphi), j=1,...,l\}.
$$

To formulate the limit theorem, we need some more notation.
%%%%%%% limit theorem
Denote by $\gamma$ the unit circle on the complex plane $\cc$, i.e., $\gamma=\{s \in \cc: |s|=1\}$. Let
$\Omega=\Omega_{1}\times \Omega_2$,
where $\Omega_1$ and $\Omega_2$ are two tori given by
$$
\Omega_{1}:=\prod_{p \in \pp} \gamma_p \quad \text{and} \quad  \Omega_2:=\prod_{m \in \nn_0}\gamma_m
$$
with $\gamma_p=\gamma$ for all  $p \in \pp$ and $\gamma_m=\gamma$ for all $m\in \nn_0$, respectively.
Since, by the Tikhonov theorem, both tori $\Omega_{1}$ and $\Omega_2$ are compact topological Abelian groups with respect to the product topology and pointwise multiplication, then the torus $\Omega$ is also a compact topological Abelian group. Therefore, on $(\Omega,{\mathcal B }(\Omega))$, a probability Haar measure $m_H$ can be defined as a product of Haar measures $m_{H1}$ and $m_{H2}$ on the spaces $(\Omega_{1}, {\mathcal{B}}(\Omega_{1}))$ and $(\Omega_2,{\mathcal B}(\Omega_2))$, respectively. This gives the probability space $(\Omega,{\mathcal B}(\Omega),m_H)$.

For $p \in \pp$, denote by $\omega_{1}(p)$ the projection of $\omega_1 \in \Omega_1$ to the coordinate space $\gamma_p$, and taking into account the decomposition of $m \in \nn$ into product of prime divisors $p^{g_p}$  define
$$
\omega_1(m)=\prod_{p \in \pp}\omega_1(p)^{g_p}.
$$
Also, for an element $\omega_2 \in \Omega_2$, its projection to the coordinate space $\gamma_m$  denote by $\omega_2(m)$, $m \in \nn_0$. Note, that both sequences $\{\omega_1(p): p \in \pp\}$ and $\{\omega_2(m): m\in \nn_0\}$ are the sequences of complex-valued random elements defined on the spaces $(\Omega_1,{\mathcal B}(\Omega_1),m_{H1})$  and $(\Omega_2,{\mathcal B}(\Omega_2),m_{H2})$, respectively.

For any open subset $G$ of the complex plane, by $H(G)$ we mean the space of
holomorphic functions on $G$ with the uniform convergence topology.
Let $D_1$ be an open subset of $D_{\varphi}$, $D_2$ be an open subset of $D_{\zeta}$,
and $\underline{H}=H(D_1)\times H(D_2)$.    Then
$$
\underline{Z}(\underline{s}):=(\varphi(s_1),\zeta(s_2,\alpha;\gb))
$$
belongs to $\underline{H}$; here $\underline{s}:=(s_1,s_2)\in D_1\times D_2$.   For $A\in\mathcal{B}(\underline{H})$, on $(\underline{H}, \mathcal{B}(\underline{H}))$, we define
$$
P_T(A)=\frac{1}{T}\meas\big\{\tau\in[0,T]:\underline{Z}(\underline{s}+i\tau)\in A\big\}
$$
with $T>0$, $\underline{s}+i\tau:=(s_1+i\tau,s_2+i\tau)$, $s_1\in D_1$, $s_2\in D_2$,
and $\meas\{A\}$ denoting the usual Lebesgue measure of the measurable set $A \subset \mathbb{R}$.

Now let $\omega:=(\omega_1,\omega_2) \in \Omega$. Define an $\underline{H}$-valued random element on the probability space $(\Omega,{\mathcal B}(\Omega),m_{H})$ by the formula
$$
\underline{Z}(\underline{s},\omega)=\left(\varphi(s_1,\omega_{1}),\zeta(s_2,\alpha,\omega_2;\gb)\right),
$$
where
\begin{align}\label{RK-for-0}
\varphi(s_1,\omega_1)=\sum_{m=1}^{\infty}\frac{c_m\omega_1(m)}{m^s}
\end{align}
and
$$
\zeta(s_2,\alpha,\omega_2;\gb)=\sum_{m=0}^{\infty}\frac{b_m\omega_2(m)}{(m+\alpha)^s}.
$$
Let $P_{\underline{Z}}$ be the distribution of this element, i.e.,
$$
P_{\underline{Z}}(A)=m_H\left\{\omega\in\Omega :\; \underline{Z}(\us,\omega)\in A\right\},
\quad A\in\mathcal{B}(\underline{H}).
$$

\begin{theorem}[Theorem~1, \cite{RK-KM-2015}]\label{RK-KM-lim}
Let $\alpha$ be a transcendental number, $0<\alpha<1$.   Then  $P_T$ converges weakly to the probability measure $P_{\underline{Z}}$ as $T \to \infty$.
\end{theorem}

This type of limit theorem is the key tool in the proof of mixed joint universality
theo\-rems.
However, though the class $\mathcal{M}$ is suitable to prove limit theorems, it is
(at least at present) too wide to discuss the universality.    Therefore we use
another class $\widetilde{S}$, which is a subclass of $\mathcal{M}$. Note that the class $\st$ was introduced by J.~Steuding (see \cite{JST-2007}).

Let $\varepsilon$ be an arbitrarily small positive number.   The Dirichlet series
$$
\varphi(s)=\sum_{m=1}^{\infty}a(m)m^{-s} \quad \text{with} \quad a(m)=O(m^{\varepsilon})
$$
(convergent absolutely in $\sigma>1$)
is said to belong to the class $\widetilde{S}$ if the following conditions hold:
it has the Euler product
$$
\varphi(s)=\prod_{n=1}^{\infty}\prod_{j=1}^l \left(1-a_j(p_n)p_n^{-s}\right)^{-1};
$$
there exists $\sigma_{\varphi}<1$ such that $\varphi(s)$ can be continued
meromorphically to $\sigma>\sigma_{\varphi}$, holomorphic there except for at most
a pole at $s=1$; the order estimate
$$
\varphi(\sigma+it)=O(|t|^{C_3+\varepsilon}),\quad C_3>0,
$$
holds for any $\sigma>\sigma_{\varphi}$, and finally there exists $\kappa>0$ for which
\begin{align}\label{prime-mean-square}
\lim_{x\to\infty}\frac{1}{\pi(x)}\sum_{p_m\leq x}|a(p_m)|^2 =\kappa
\end{align}
holds, where $\pi(x)$ denotes the number of prime numbers up to $x$.

For $\varphi\in\widetilde{S}$, let $\sigma^*$ be the infimum of all $\sigma_1$
for which
$$
\frac{1}{2T}\int_{-T}^T |\varphi(\sigma+it)|^2 dt \sim \sum_{m=1}^{\infty}
\frac{|a(m)|^2}{m^{2\sigma}}
$$
holds for any $\sigma\geq\sigma_1$.    Then $1/2\leq\sigma^*<1$.    (And so,
choosing $\sigma_0=\sigma^*+\varepsilon$, we find that
$\widetilde{S}\subset\mathcal{M}$.)

For convenience, let $D(a,b):=\{s \in \mathbb{C}: a <\sigma <b\}$ for every $a<b$.

\begin{theorem}[Theorem~2, \cite{RK-KM-2015}]\label{RK-KM-univ}
Suppose $\varphi\in\widetilde{S}$.
Let $\alpha$ be a transcendental number, $0<\alpha <1$.
Let $K_1$ be a compact subset of the strip $D\big(\sigma^*,1\big)$, $K_2$ be a
compact subset of the strip $D\big(1/2,1\big)$, both with connected complements.
Suppose that the function $f_1(s)$ is continuous non-vanishing on $K_1$, analytic in
the interior of $K_1$, while the function $f_2(s)$ is continuous on $K_2$, analytic
in the interior of $K_2$. Then, for every $\varepsilon>0$,
\begin{align*}
\liminf_{T \to \infty}\frac{1}{T}\meas \bigg\{\tau \in [0,T]: & \sup_{s \in K_1}|\varphi(s+i\tau)-f_1(s)|<\varepsilon, \cr
&\sup_{s \in K_2}|\zeta(s+i\tau,\alpha;\gb)-f_2(s)|<\varepsilon\bigg\}>0.
\end{align*}
\end{theorem}

A further generalization was done
in \cite{RK-KM-2017-BAMS}.    Let $0<\alpha_j<1$, $j=1,\ldots,r$, $l(j)\in\nn$,
and let $\gb_{jl}=\{b_{mjl}\in\cc:m\in\nn_0\}$ be a periodic sequence of complex numbers with the minimal period $k_{jl}$, $j=1,...,r$, $l=1,...,l(j)$.    Denote by
$k_j$ the least common multiple of $k_{j1},\ldots,k_{jl(j)}$, and let
\begin{align*}
B_j=\begin{pmatrix}
b_{1j1} & b_{1j2} &... & b_{1jl(j)}\cr
b_{2j1} & b_{2j2} &... & b_{2jl(j)}\cr
... & ... &... & ...\cr
b_{k_j j1} & b_{k_j j2} &... & b_{k_j jl(j)}\cr
\end{pmatrix}.
\end{align*}

\begin{theorem}[Theorem~4.2, \cite{RK-KM-2017-BAMS}]\label{RK-KM-gen-univ}
Suppose $\varphi\in\widetilde{S}$, the numbers
$\alpha_1,\ldots,\alpha_r$  are algebraically independent over $\qq$,
${\rm rank} B_j=l(j)$, $j=1,...,r$.
Let $K_1$, $f_1(s)$ be as in Theorem~\ref{RK-KM-univ}, while
$K_{2jl}$ be
compact subsets of the strip $D\big(1/2,1\big)$ with connected complements, and
the functions $f_{2jl}(s)$ are continuous on $K_{2jl}$, analytic
in the interior of $K_{2jl}$. Then, for every $\varepsilon>0$,
\begin{align*}
\liminf_{T \to \infty}\frac{1}{T}\meas \bigg\{\tau \in [0,T]: & \sup_{s \in K_1}|\varphi(s+i\tau)-f_1(s)|<\varepsilon, \cr
&\max_{1\leq j\leq r}\max_{1\leq l\leq l(j)}\sup_{s \in K_{2jl}}|\zeta(s+i\tau,\alpha_j;\gb_{jl})-f_{2jl}(s)|<\varepsilon\bigg\}>0.
\end{align*}
\end{theorem}

Note that a similar result was independently announced by R.~Macaitien\.e
(see \cite{RM-2015}).
% with a very brief outline of the proof.    In her work $\varphi(s)$
%is replaced by a zeta-function belonging to the Selberg class, satisfying
%additional mean square formula of the form \eqref{prime-mean-square}.

%%%%%%%%%%%%%%%%%%%%%%%%%%%%%%%%%%%%%%%%%%%%%%%%%%%%%%%%%%%%%%%%%%%%%%%%%%%%
\section{Statement of results: the discrete case}\label{sor}
%%%%%%%%%%%%%%%%%%%%%%%%%%%%%%%%%%%%%%%%%%%%%%%%%%%%%%%%%%%%%%%%%%%%%%%%%%%%

In the previous section, we discussed the universality in which the shif\-ting
parameter $\tau$ is moving continuously.     It is also possible to consider the
situation where the shif\-ting parameter takes only discrete values.    This type of
universality is called the discrete universality.

It is to be stressed that the arithmetic nature of the shifting parameter plays a
role in the proof of discrete universality.    Therefore we may say that the
discrete universality is more arithmetic phenomenon.
Also, from the viewpoint of applications, in pure theoretical or in practical sense, the discrete universality is often more suitable (for the details, see \cite{KM-2015}).

A discrete analogue of Theorem \ref{RK-KM-univ} was shown in
\cite{RK-KM-2017-Pal}.    Let $h>0$, and put
\begin{align*}
L(\alpha,h):=\left\{\log p: p \in \pp\right\}\cup\left\{\log (m+\alpha): m \in \nn_0\right\}\cup\left\{\frac{2\pi}{h}\right\}.
\end{align*}

\begin{theorem}[Theorem~3, \cite{RK-KM-2017-Pal}]\label{RK-KM-disc-univ}
Suppose $\varphi\in\widetilde{S}$, and that the elements of the set $L(\alpha,h)$ are linearly independent over $\qq$.
Let $K_1$, $K_2$, $f_1(s)$, $f_2(s)$ be the same as in Theorem \ref{RK-KM-univ}.
Then, for every $\varepsilon>0$,
\begin{align*}
\liminf_{N \to \infty}\frac{1}{N+1}\# \bigg\{0\leq k\leq N: & \sup_{s \in K_1}|\varphi(s+ikh)-f_1(s)|<\varepsilon, \cr
&\sup_{s \in K_2}|\zeta(s+ikh,\alpha;\gb)-f_2(s)|<\varepsilon\bigg\}>0.
\end{align*}
Here and in what follows, $\# \{A\}$ denotes the cardinality of the set $A$.
\end{theorem}

In the above theorem, the common difference of arithmetic progressions is the same $h$
for both of the zeta-functions.
However it is possible to consider the case when common differences of progressions for
two zeta-functions are different.   The aim of the present paper is to prove
discrete universality results in such a situation, and further give certain
generalizations.

Let $h_1$ and $h_2$ be positive numbers defining common differences of arithmetic progressions, and $N\in\nn$.

Define the set
\begin{align*}
L(\alpha,h_1,h_2):=\left\{h_1 \log p: p \in \pp\right\}\cup\left\{h_2 \log (m+\alpha): m \in \nn_0\right\}\cup\left\{\pi\right\}.
\end{align*}

First we study $P_N$, on $(\underline{H},{\mathcal B}(\underline{H}))$, defined by
the formula
$$
P_N(A):=\frac{1}{N+1}\#\big\{0 \leq k \leq N: \underline{Z}(s_1+ikh_1,s_2+ikh_2)\in A\big\}, \quad A \in {\mathcal B}(\underline{H}).
$$

\begin{theorem}\label{RK-KM-new-lim}
Let $\varphi\in\mathcal{M}$.
Suppose the elements of the set $L(\alpha,h_1,h_2)$ are linearly independent over $\qq$. Then $P_N$ converges weakly to
%the distribution of the random element $\underline{Z}(s,\omega)$
$P_{\underline{Z}}$ as $N$ tends to infinity.
\end{theorem}

A generalization of Theorem \ref{RK-KM-disc-univ} is as follows.

\begin{theorem}\label{RK-KM-new-disc-univ}
Suppose that $\varphi\in\widetilde{S}$, and that the elements of the set $L(\alpha,h_1,h_2)$ are li\-near\-ly independent over $\qq$.  Let $K_1$, $K_2$, $f_1(s)$ and  $f_2(s)$ be as in Theorem~\ref{RK-KM-univ}. Then, for every $\varepsilon>0$,
\begin{align*}
\liminf_{N \to \infty}\frac{1}{N+1}\#
\bigg\{0 \leq k \leq N : & \sup_{s \in K_1}|\varphi(s+i k h_1)-f_1(s)|<\varepsilon, \cr
&\sup_{s \in K_2}|\zeta(s+i k h_2,\alpha;\gb)-f_2(s)|<\varepsilon\bigg\}>0.
\end{align*}
\end{theorem}

These two theorems will be proved in the next section.
For the proof, we will use the method of limit theorems in the sense of weakly convergent probability measures in the space of analytic functions developed by A.~Laurin\v cikas (see, for example, \cite{AL-1996}) and probabilistic approach introduced by B.~Bagchi (see \cite{BB-1981}).

\begin{remark}
The first result on the discrete universality was obtained  by A.~Reich (see \cite{AR-1980}), where Dedekind zeta-functions were studied.
Discrete analogues of the Mishou theorem (or the mixed universality theorem, see \cite{HM-2007}) for $\zeta(s)$ and $\zeta(s,\alpha)$ were proved by E.~Buivydas and A.~Laurin\v cikas (see \cite{EB-AL-2015-RJ}, \cite{EB-AL-2015-LMJ}).
In \cite{EB-AL-2015-RJ}, the common difference of arithmetic progressions for both the functions is the same, while in \cite{EB-AL-2015-LMJ}, common differences are not necessarily the same for both the functions.   Therefore the above Theorem \ref{RK-KM-disc-univ} and Theorem
\ref{RK-KM-new-disc-univ} are generalizations of the results in
\cite{EB-AL-2015-RJ} and \cite{EB-AL-2015-LMJ}, respectively.

Finally note that the first attempt to prove mixed universality theorem for a Dirichlet $L$-func\-tion and a periodic Hurwitz zeta-function was done by the first-named author in 2009  (see \cite{RK-2009}). But in the proof there exists a gap.  The comments how to fulfill this gap and how to prove a mixed universality theorem for a certain modified
zeta-function can be found in the authors' works (see \cite{RK-KM-2017-Pal}, \cite{RK-KM-draft}).
\end{remark}

%%%%%%%%%%%%%%%%%%%%%%%%%%%%%%%%%%%%%%%%%%%%%%%%%%%%%%%%%%%%%%%%%%%%%%%%%%%%%%
\section{Proof of Theorems~\ref{RK-KM-new-lim} and \ref{RK-KM-new-disc-univ}}
%%%%%%%%%%%%%%%%%%%%%%%%%%%%%%%%%%%%%%%%%%%%%%%%%%%%%%%%%%%%%%%%%%%%%%%%%%%%%%

The proof of Theorem~\ref{RK-KM-new-lim} goes along the same line as that of Theorem~4 in \cite{EB-AL-2015-LMJ}. Therefore we present only some essential points of the proof, as a series of lemmas.

Let $\varphi\in\mathcal{M}$.
First we show a discrete mixed limit theorem for absolutely convergent Dirichlet series.
For a fixed number $\sigma_0^*>\frac{1}{2}$  and $n \in \nn$, we put
$$
v_1(m,n)=\exp\left(-\left(\frac{m}{n}\right)^{\sigma_0^*}\right),\quad m\in\nn,
$$
$$
v_2(m,n,\alpha)=\exp\left(-\left(\frac{m+\alpha}{n+\alpha}\right)^{\sigma_0^*}\right),
\quad m\in\nn_0,
$$
and define
$$
\underline{Z}_n({\underline s})=\left(\varphi_n(s_1),\zeta_n(s_2,\alpha;\gb)\right)
$$
with
$$
\varphi_n(s_1)=\sum_{m=1}^{\infty}\frac{c_mv_1(m,n)}{m^{s_1}}\quad \text{and} \quad
\zeta_n(s_2,\alpha;\gb)=\sum_{m=0}^{\infty}\frac{b_mv_2(m,n,\alpha)}{(m+\alpha)^{s_2}}.
$$
These series converge absolutely for $\Re s_j>\frac{1}{2}$, $j=1,2$. Moreover, for
$\omega\in \Omega$, let
$$
\underline{Z}_n(\underline{s},\omega)=\left(\varphi_n(s_1,\omega_1),\zeta_n(s_2,\alpha,\omega_2;\gb)\right)
$$
with
$$
\varphi_n(s_1,\omega_1)=\sum_{m=1}^{\infty}\frac{c_m\omega_1(m) v_1(m,n)}{m^{s_1}}\ \text{ and } \
\zeta_n(s_2,\alpha,\omega_2;\gb)=\sum_{m=0}^{\infty}\frac{b_m \omega_2(m)v_2(m,n,\alpha)}{(m+\alpha)^{s_2}}.
$$
The series $\varphi_n(s_1,\omega_1)$ and $\zeta_n(s_2,\alpha,\omega_2;\gb)$ also converge for $\Re s_j>\frac{1}{2}$, $j=1,2$.

For $A \in {\mathcal B}(\underline{H})$, define
\begin{align*}
P_{N,n}(A):=\frac{1}{N+1}\#\bigg\{0 \leq k \leq N: \underline{Z}_n(s_1+ikh_1,s_2+ikh_2) \in A
\bigg\}
\end{align*}
and, for a fixed element ${\widehat \omega}:=\big({\widehat \omega}_1,{\widehat \omega}_2\big)\in \Omega$,
\begin{align*}
P_{N,n,{\widehat \omega}}(A):=\frac{1}{N+1}\#\bigg\{0 \leq k \leq N: \underline{Z}_n(s_1+ikh_1,s_2+ikh_2,{\widehat\omega}) \in A\bigg\}.
\end{align*}

\begin{lemma}\label{RK-le-1}
Suppose that the elements of the set $L(\alpha,h_1,h_2)$ are linearly independent over $\qq$. Then both measures $P_{N,n}$ and $P_{N,n,{\widehat \omega}}$ converge weakly to the same probability measure, which we denote by $P_n$, on $(\underline{H},$ ${\mathcal B}(\underline{H}))$ as $N \to \infty$.
\end{lemma}

In the proof of this lemma, a crucial role is played by the weak convergence of the measure
\begin{align*}
\frac{1}{N+1}\#\left\{0 \leq k \leq N: \big(\big(p^{-ikh_1}: p \in \pp\big),\big((m+\alpha)^{-ikh_2}: m\in \nn_0\big)\big)\in A\right\}, \quad A \in {\mathcal B}(\Omega),
\end{align*}
to the Haar measure $m_H$ as $N \to \infty$.
This is Lemma 1 of \cite{EB-AL-2015-LMJ},
whose proof depends on the linear independence of the elements of the set $L(\alpha,h_1,h_2)$.

Next, we approximate by mean the tuples  $\underline{Z}({\underline s})$ and $\underline{Z}({\underline s},\omega)$ by the tup\-les $\underline{Z}_n(\underline{s})$ and $\underline{Z}_n(\underline{s},\omega)$, respectively.
Let ${\underline \varrho}$ be the metric on $\underline{H}$ defined in Section 3 of
\cite{EB-AL-2015-LMJ} (or Section 2 of \cite{RK-KM-2017-Pal}).

\begin{lemma}\label{RK-le-2}
Suppose that the elements of $L(\alpha,h_1,h_2)$ are linearly independent over $\qq$. The relations
\begin{align*}
\lim_{n \to \infty}\limsup_{N \to \infty}\frac{1}{N+1}\sum_{k=0}^{N}
{\underline \varrho}\big(\underline{Z}(s_1+ikh_1,s_2+ikh_2),\underline{Z}_n(s_1+ikh_1,s_2+ikh_2)\big)=0
\end{align*}
and, for almost all $\omega \in \Omega$,
\begin{align*}
\lim_{n \to \infty}\limsup_{N \to \infty}\frac{1}{N+1}\sum_{k=0}^{N}
{\underline \varrho}\big(\underline{Z}(s_1+ikh_1,s_2+ikh_2,\omega),\underline{Z}_n(s_1+ikh_1,s_2+ikh_2,\omega)\big)=0
\end{align*}
are true.
\end{lemma}

This can be obtained in a way similar to Lemma~3 of \cite{RK-KM-2017-Pal}. Note that the linear independence of the elements of the set $L(\alpha,h_1,h_2)$ over
$\qq$ implies the linear independence of $\{\log (m+\alpha): m \in \nn_0\}$, which is used in the proof of the second part of the above lemma (similar to the argument in the proof of
Lemma 4 in  \cite{EB-AL-2015-RJ}).

In the third step, we examine one more probability measure, for $A \in {\mathcal B}(\underline{H})$, defined by
$$
P_{N,\omega}(A):=\frac{1}{N+1}\#\bigg\{0\leq k \leq N: \underline{Z}(s_1+ikh_1,s_2+ikh_2,\omega)\in A\bigg\}
$$
for almost all $\omega \in \Omega$.

\begin{lemma}\label{RK-le-3}
Assume that the elements of the set $L(\alpha,h_1,h_2)$ are linearly independent over $\qq$. Then the measures $P_N$ and $P_{N,\omega}$ both converge weakly to the same probability measure, which we denote by $P$, on $(\underline{H},{\mathcal B}(\underline{H}))$ as $N \to \infty$.
\end{lemma}

This is an analogue of Lemma~5 in \cite{EB-AL-2015-RJ} or Lemma~5 in
\cite{EB-AL-2015-LMJ}.
In the proof, Lemma~\ref{RK-le-1}, the both relations of Lemma~\ref{RK-le-2},  together with Gallagher's lemma (see \cite{HLM-1971}) and Prokhorov's theorem (see \cite{PB-1968}) are applied.

The final step of the proof of the discrete mixed limit theorem is the identification of the measure $P$ of Lemma~\ref{RK-le-3}.

\begin{lemma}\label{RK-le-4}
The probability measure $P$ coincides with the probability measure $P_\uZ$.
\end{lemma}

This is analogous to Theorem~4 of \cite{EB-AL-2015-LMJ}.
For the proof, the elements from ergodic theo\-ry are used. For $\omega \in \Omega$, we consider the one-parameter group of measurable measure preserving transformation of the torus $\Omega$ defined by \begin{align*}
\Phi_{\alpha,h_1,h_2}:=a_{\alpha,h_1,h_2}\omega, \quad \omega \in \Omega,
\end{align*}
where
$a_{\alpha,h_1,h_2}:=\big((p^{-ih_1}:p \in \pp),((m+\alpha)^{-ih_2}:m \in \nn_0)\big)$. Since the elements of the set $L(\alpha,h_1,h_2)$ are linearly independent over $\qq$, we can show that  the group $\Phi_{\alpha,h_1,h_2}$ is ergodic.  Therefore, using this fact together with Lemma~\ref{RK-le-3} and the Birkhoff-Khintchine theorem (see \cite{HCR-MRL-1967}), we obtain Lemma \ref{RK-le-4}.

Collecting all of these lemmas, we prove the assertion of Theorem~\ref{RK-KM-new-lim}.

%%%%%%%%%%%%%%%%%%%%%%%%%%%%%%%%%%%%%%%%%%%%%%%%%%%%%%%%%%%%%%%%%%%%%%%%%%%
%%%%%\section{Proof of Theorem~\ref{RK-KM-new-disc-univ}}
%%%%%%%%%%%%%%%%%%%%%%%%%%%%%%%%%%%%%%%%%%%%%%%%%%%%%%%%%%%%%%%%%%%%%%%%%%%%

\bigskip
Next we proceed to the proof of Theorem \ref{RK-KM-new-disc-univ}.

Assume that $\varphi\in\widetilde{S}$, and let $K_1,K_2,f_1(s)$ and $f_2(s)$ be as in the
statement of Theorem \ref{RK-KM-new-disc-univ}.    Let $M>0$ be sufficiently large such that
$K_1$ is included in
$$
D_M=\{s:\sigma_0<\sigma<1,|t|<M\}.
$$
Since $\varphi\in\widetilde{S}$, we see that
$D_{\varphi}=\{s:\sigma>\sigma_0,\sigma\neq 1\}$, so $D_M\subset D_{\varphi}$.
Also we can find $T>0$ such that $K_2$ is included in
$$
D_T=\{s:1/2<\sigma<1,|t|<T\}.
$$
We choose $D_1=D_M$ and $D_2=D_T$ and consider
an explicit form of the support $S_{\underline{Z}}$ of the probability measure $P_{\underline{Z}}$.
Define
$S_{\varphi}:=\{f \in H(D_M): f(s)\not =0 \ \text{for} \ D_M,\ \text{or} \ f(s) \equiv 0\}$.

\begin{lemma}\label{RK-le-5}
Suppose that the elements of the set $L(\alpha,h_1,h_2)$ are linearly independent over $\qq$. Then the support of the measure $P_{\underline{Z}}$ is the set
$S_{\underline{Z}}:=S_{\varphi}\times H(D_T)$.
\end{lemma}

\begin{proof}
The support of the random variable $\varphi(s_1,\omega_1)$ is $S_{\varphi}$ (see Lemma~5.12 of \cite{JST-2007}). Note that here we essentially need the assumption that
$\varphi\in\widetilde{S}$ (see Remark~4.4 of \cite{RK-KM-2015}).
Since the elements of the set $L(\alpha,h_1,h_2)$ are linearly independent, it follows that $\{\log(m+\alpha): m \in \nn_0\}$ is linearly independent. Therefore the support of the random element $\zeta(s_2,\alpha,\omega_2;\ga)$ is the whole of $H(D_T)$ (see \cite{AL-2006}).

The space $H(D)$ is separable, then we have that
${\mathcal B}(\underline{H})={\mathcal B}(H(D_M))\times {\mathcal B}(H(D_T))$ (see \cite{PB-1968}), and it suffices to study the measure $P_\uZ$  on the sets $A$, where $A=A_1 \times A_2$ with  $A_1\in H(D_M)$ and $A_2 \in H(D_T)$.

In view of the definition of the measure $m_H$ as a product of the measures $m_{H1}$ and $m_{H2}$ (see Section~\ref{intro}), we obtain
\begin{align*}
P_\uZ(A)&=m_{H1}\{\omega_1 \in \Omega_1: \varphi(s,\omega_1) \in A_1\}\times m_{H2}\{\omega_2 \in \Omega_2: \zeta(s,\alpha,\omega_2;\gb) \in A_2\}.
\end{align*}
This shows that $P_\uZ(A)=1$ %is the set $S_\uZ$
if and only if
$m_{H1}\{\omega_1 \in \Omega_1: \varphi(s,\omega_1) \in A_1\}=1$ and $m_{H2}\{\omega_2 \in \Omega_2: \zeta(s,\alpha,\omega_2;\gb) \in A_2\}=1$.
Therefore the minimal set $A$ such that $P_\uZ(A)=1$ is the set $A=S_\uZ$.
\end{proof}

The proof of Theorem~\ref{RK-KM-new-disc-univ}  is a consequence of Theorem~\ref{RK-KM-new-lim}, Lemma~\ref{RK-le-5} and  the Mergelyan theorem, which asserts that any continuous function $f(s)$ on a compact subset $K \subset \cc$ with connected complement  which is analytic in the interior of $K$ is approximable uniformly on $K$ by
certain suitably chosen polynomials  $p(s)$ (see \cite{SNM-1952}).
The argument is standard and the same as in Section 4 of \cite{RK-KM-2017-Pal}, so we omit the details.

%%%%%%%%%%%%%%%%%%%%%%%%%%%%%%%%%%%%%%%%%%%%%%%%%%%%%%%%%%%%%%%%%%%%%%%%%%%%%
\section{Generalizations}
%%%%%%%%%%%%%%%%%%%%%%%%%%%%%%%%%%%%%%%%%%%%%%%%%%%%%%%%%%%%%%%%%%%%%%%%%%%%%%

As we noted in Section~\ref{sor}, in this paper we will give certain generalizations of Theo\-rems~\ref{RK-KM-new-lim} and \ref{RK-KM-new-disc-univ}, namely, we extend to the case of a collection of periodic Hurwitz zeta-functions.  In this section, we will focus on the study of the joint discrete mixed universality for such collections of zeta-functions with periodic coefficients whose common differences are not necessarily the same as each other. This is the main novelty of the present paper, since by the authors' knowledge, there are several former articles which treat such results in the continuous case (see \cite{AL-2010}, \cite{RK-BRK-2018}), but only few work in the discrete case. 
For some special $\varphi$, there are papers \cite{RK-2018}, \cite{AL-2018}
in the discrete case.
In \cite{RK-2018}, distinct common differences are treated but for a single periodic 
Hurwitz zeta-functions, while in \cite{AL-2018} a collection of periodic Hurwitz
zeta-functions are considered but with the same common difference. 

Let $\gb_j=\{b_{mj}: m \in \nn_0\}$ be a periodic sequence of complex numbers with a minimal period $l_j\in \nn$, $\alpha_j$ be  a fixed parameter, $0<\alpha_j<1$, $j=1,...,r$. Suppose that $\zeta(s,\alpha_j;\gb_j)$ is the corresponding periodic Hurwitz zeta-function for $j=1,...,r$.

Now let $\uA:=\big(\alpha_1,...,\alpha_r\big)$, and, for $h_1>0, h_{21}>0,...,h_{2r}>0$, let $\uhh:=\big(h_1,h_{21},...,h_{2r}\big)$. Define the set
\begin{align*}
L(\uA,\uhh) :=&L(\alpha_1,...,\alpha_r,h_1,h_{21},...,h_{2r})\cr
=&\{h_1 \log p: p \in \pp\}\bigcup_{j=1}^{r}\{h_{2j}\log(m+\alpha_j): m\in \nn_0\}\cup\{\pi\}.
\end{align*}

\begin{theorem}\label{RK-KM-new-joint-disc-univ}
Let $\varphi\in\widetilde{S}$.   Suppose that the elements of the set $L(\uA,\uhh)$ are linearly independent over $\qq$, and $K_1$ and $f_1(s)$ satisfy the conditions of Theorem~\ref{RK-KM-univ}. Let $K_{2j}$  be compact subsets of the strip $D\big(\frac{1}{2},1\big)$ with connected complements,  $f_{2j}(s)$ be continuous functions in $K_{2j}$  and analytic in the interior of $K_{2j}$ for  all $j=1,...,r$, respectively. Then, for every $\varepsilon>0$, it holds that
\begin{align*}
\liminf_{N \to \infty}\frac{1}{N+1}\#\bigg\{
0\leq k \leq N: & \sup_{s \in K_1}|\varphi(s+ikh_1)-f_0(s)|<\varepsilon, \cr
& \sup_{1 \leq j \leq r}\sup_{s \in K_{2j}}|\zeta(s+ikh_{2j},\alpha_j;\gb_j)-f_{2j}(s)|<\varepsilon
\bigg\}>0.
\end{align*}
\end{theorem}

\begin{remark}
It is possible to generalize further the above theorem to the case of the collection of
$\zeta(s,\alpha_j;\gb_{jl})$ as in Theorem \ref{RK-KM-gen-univ}.
\end{remark}

For the proof of Theorem~\ref{RK-KM-new-joint-disc-univ}, we use a functional mixed joint discrete limit theorem in the sense of weakly convergent probability measures for the vector $\big(\varphi(s),\zeta(s,\alpha_1;\gb_1),$ $...,\zeta(s,\alpha_r;\gb_r)\big)$ in the space of analytic functions. This theorem generalizes Theorem~\ref{RK-KM-new-lim}.

Let $\varphi\in\mathcal{M}$.
For the quantities $\omega_1$,  $\Omega_1$, $m_{H1}$,  $\omega_2$,  $\Omega_2$ and $m_{H2}$ we keep the same notations and mea\-nings as in the previous sections.

Suppose that $\Omega_{2j}=\Omega_2$, for all $j=1,...,r$, and put
$
\uO:=\Omega_1\times \Omega_{21} \times... \times\Omega_{2r}.
$
The torus $\uO$ is a compact topological Abelian group. Therefore, on $(\uO,{\mathcal B}(\uO))$, there exists a pro\-ba\-bi\-li\-ty Haar measure $m_{Hr}$ defined as the product  $m_{H1}\times m_{H21}\times\cdots\times m_{H2r}$, where $m_{H2j}$ is the Haar measure defined on the space $(\Omega_{2j},{\mathcal B}(\Omega_{2j}))$ for all $j=1,...,r$ (see \cite{JST-2007}). This leads to the probability space $(\uO,{\mathcal B}(\uO),m_{Hr})$. Denote by $\omega_1(p)$ (resp. $\omega_{2j}(m)$) the projection of $\omega_1 \in \Omega_1$ (resp. $\omega_{2j}\in\Omega_{2j}$) onto the coordinate space $\gamma_p$, $p\in\pp$ (resp. $\gamma_m$, $m \in \nn_0$), and further define $\omega_1(m)$ as before.

Elements of the torus $\uO$ are generally written as $\uom:=\big(\omega_1,\omega_{21},...,\omega_{2r}\big)$, but sometimes we use ${\widehat \uom}:=\big({\widehat \omega_1},{\widehat \omega_{21}},...,{\widehat \omega}_{2r}\big)\in \uO$ for a fixed element.

Let $D_1$ be an open subset of $D_{\varphi}$, $D_{2j}$ be an open subset of
$D_{\zeta_j}$, where $\zeta_j=\zeta(s_{2j},\alpha_j;\gb_j)$, $j=1,..., r$, and put
$$
\uH_r=H(D_1)\times H(D_{21})\times ...\times H(D_{2r}).
$$
Let $\us=:(s_1,s_{21},...,s_{2r})\in D_1\times D_{21}\times\cdots\times D_{2r}$.    On the probability space $(\uO,{\mathcal B}(\uO),m_{Hr})$, define an $\uH_r$-valued random element by the formula
$$
\underline{Z}(\us,\uA,\uom):=\big(\varphi(s_1,\omega_1),\zeta(s_{21},\alpha_1,\omega_{21};\gb_1),...,\zeta(s_{2r},\alpha_r,\omega_{2r};\gb_r)\big),
$$
where
$$
\zeta(s_{2j},\alpha_j,\omega_{2j};\gb_j)=\sum_{m=0}^{\infty}\frac{b_{mj}\omega_{2j}(m)}{(m+\alpha_j)^{s_{2j}}}, \quad j=1,...,r,
$$
and $\varphi(s_1,\omega_1)$ is given by  \eqref{RK-for-0}
(for the details, see \cite{AL-2010}, \cite{RK-KM-2017-BAMS}).
Denote the distribution of the random element $\underline{Z}(\us,\uA,\uom)$ by $\uuZ$, i.e.,
\begin{align*}
\uuZ(A):=m_{Hr}\big\{\uom \in \uO: \underline{Z}(\us,\uA,\uom) \in A\big\}, \quad A \in {\mathcal B}(\uH_r).
\end{align*}

Now we generalize the previous notation of $\uZ({\underline s})$ to define
$$
\uZ(\us):=\big(\varphi(s_1),\zeta(s_{21},\alpha_1;\gb_1),...,\zeta(s_{2r},\alpha_r;\gb_r)\big).
$$
We will show the following joint discrete mixed theorem.

\begin{theorem}\label{RK-KM-new-joint-lim}
Let $\varphi\in\mathcal{M}$.
Suppose that the elements of the set $L(\uA,\uhh)$ are linearly independent over $\qq$. Then, the measure ${\up}_N$ defined by  
\begin{align*}
{\up}_N(A):=\frac{1}{N+1}\#\bigg\{&0 \leq k \leq N:  \big(\uZ(s_1+ikh_1,s_{21}+ikh_{21},...,s_{2r}+ikh_{2r})\big)\in A\bigg\},
\end{align*}
$A \in {\mathcal B}(\uH_r)$,
converges weakly to $\uuZ$ as $N \to \infty$.
\end{theorem}

\begin{proof}
The proof of the above theorem is quite similar to that for Theorem~\ref{RK-KM-new-lim} (which is the case $r=1$). We give a sketch with a special emphasis to the complicated places in the proof.

As natural generalizations of $\uZ(s_1+ikh_1,s_2+ikh_2)$, $\uZ_n(s_1+ikh_1,s_2+ikh_2)$, $\uZ(s_1+ikh_1,s_2+ikh_2,\omega)$ and $\uZ_n(s_1+ikh_1,s_2+ikh_2,\omega)$, we define $\uZ(s_1+ikh_1,s_{21}+ikh_{21},...,s_{2r}+ikh_{2r})$, $\uZ_n(s_1+ikh_1,s_{21}+ikh_{21},...,s_{2r}+ikh_{2r})$, $\uZ(s_1+ikh_1,s_{21}+ikh_{21},...,s_{2r}+ikh_{2r},\uom)$ and $\uZ_n(s_1+ikh_1,s_{21}+ikh_{21},...,s_{2r}+ikh_{2r},\uom)$, respectively, in an obvious manner.    We can further define the quantities $\up_{N,n}$, $\up_{N,n,{\widehat \uom}}$ and $\up_{N,\uom}$ similarly to
$P_{N,n}$, $P_{N,n,{\widehat \omega}}$ and $P_{N,\omega}$, respectively.

Similar to Lemma~\ref{RK-le-1}, we can show that both measures $\up_{N,n}$ and $\up_{N,n,{\widehat \uom}}$ converge weakly to a certain probability measure $\up_n$ as $N \to \infty$. In the proof, a key role is played by a limit theorem on the torus $\uO$, which we separate as the lemma.

\begin{lemma}\label{RK-le-j-torus}
Suppose that the elements of the set $L(\uA,\uhh)$ are linearly independent over $\qq$. Then, on $(\uO,{\mathcal B}(\uO))$, the pro\-ba\-bi\-li\-ty measure
\begin{align*}
{\underline Q}_N(A):=\frac{1}{N+1}\#\big\{0\leq k \leq N:
\big(&
\big(p^{-ikh_1}: p \in \pp\big),\big((m+\alpha_1)^{-ikh_{21}}: m \in \nn_0
\big),...,\cr &
\big((m+\alpha_r)^{-ikh_{2r}}: m \in \nn_0
\big)
\big) \in A
\big\}, \quad A \in {\mathcal B}(\uO),
\end{align*}
converges weakly to the Haar measure $m_{Hr}$ as $N \to \infty$.
\end{lemma}

\begin{proof}
For the proof, we use the Fourier transformation method (see \cite{AL-1996}). The dual group of $\uO$ is isomorphic to the group
$$
G:=\bigg(
\bigoplus_{p \in \pp}\zz_{p}
\bigg)
\bigoplus_{j=1}^{r}
\bigg(
\bigoplus_{m \in \nn_0}\zz_{mj}
\bigg),
$$
where $\zz_p=\zz$ for all $p \in \pp$ and $\zz_{mj}=\zz$ for all $m \in \nn_0$, $j=1,...,r$. We take an element $({\underline k},{\underline l}_1,...,{\underline l}_r):=\big(
(k_p: p \in \pp),(l_{m1}: m\in \nn_0),...,(l_{mr}: m \in \nn_0)
\big)\in G$, where only a finite number of integers $k_p$ and $l_{m1},...,l_{mr}$ are nonzero, and acts on $\uO$ as
\begin{align*}
(\omega_1,\omega_{21},...,\omega_{2r})\to
(\omega_1^{\underline k},\omega_{21}^{{\underline l}_1},...,\omega_{2r}^{{\underline l}_r})=
\prod_{p \in \pp}\omega_1^{k_p}(p)\prod_{j=1}^{r}\prod_{m \in \nn_0}\omega_{2j}^{l_{mj}}(m).
\end{align*}

Denote by  ${\underline g}_N({\underline k},{\underline l}_1,...,{\underline l}_r)$ the Fourier transform of the measure ${\underline Q}_N(A)$ for $({\underline k},{\underline l}_1,...,{\underline l}_r) \in G$, i.e.,
$$
{\underline g}_N({\underline k},{\underline l}_1,...,{\underline l}_r)=\int_{\uO}\bigg(
\prod_{p \in \pp}\omega_1^{k_p}(p)\prod_{j=1}^{r}\prod_{m \in \nn_0}\omega_{2j}^{l_{mj}}(m)
\bigg) d {\underline Q}_N,
$$
where as above only finite number of integers $k_p$ and $l_{mj}$ are non-zero. Thus, by the definition of ${\underline Q}_N$,
\begin{align}\label{RK-for-1t}
{\underline g}_N({\underline k},{\underline l}_1,...,{\underline l}_r)=&
\frac{1}{N+1}\sum_{k=0}^{N}\prod_{p \in \pp}p^{-ikh_1 k_p}
\prod_{j=1}^{r}\prod_{m \in \nn_0}(m+\alpha_j)^{-ikh_{2j} l_{mj}}\cr
=&\frac{1}{N+1}\sum_{k=0}^{N}\exp\bigg\{
-ik\bigg(
\sum_{p \in \pp}h_1k_p\log p
+
\sum_{j=1}^{r}\sum_{m \in \nn_0}h_{2j}l_{mj}\log(m+\alpha_j)\bigg)
\bigg\}. \cr
\end{align}
%By the assumption of the lemma, the elements of the set
%$$
%\{\log p: p \in \pp\}\cup\bigcup_{j=1}^{r}\{\log (m+\alpha_j): m \in \nn_0\}
%$$
%are linearly independent over $\qq$. Therefore
Obviously
\begin{align}\label{RK-for-2t}
\sum_{p \in \pp}h_1k_p\log p
+
\sum_{j=1}^{r}\sum_{m \in {\nn_0}}h_{2j}l_{mj}\log(m+\alpha_j)=0
\end{align}
if ${\underline k}={\underline 0}$, ${\underline l}_{1}={\underline 0},...,{\underline l}_{r}={\underline 0}$.
On the other hand, we observe that
\begin{align}\label{RK-for-3t}
\exp\bigg\{
-i\bigg(
\sum_{p \in \pp}h_1k_p\log p
+
\sum_{j=1}^{r}\sum_{m \in \nn_0}h_{2j}l_{mj}\log(m+\alpha_j)\bigg)
\bigg\}\not =1
\end{align}
for $({\underline k},{\underline l}_1,...,{\underline l}_r)\not =({\underline 0},{\underline 0},...,{\underline 0})$.
Indeed, if \eqref{RK-for-3t} is false, then
\begin{align*}
\sum_{p \in \pp}h_1k_p\log p
+
\sum_{j=1}^{r}\sum_{m \in \nn_0}h_{2j}l_{mj}\log(m+\alpha_j)=2 \pi a
\end{align*}
with some $a \in \zz$. But this contradicts to the linear independence of the set $L(\uA,\uhh)$.
Therefore, from \eqref{RK-for-2t} and \eqref{RK-for-3t} together with \eqref{RK-for-1t} we get
\begin{align*}
{\underline g}_N({\underline k},{\underline l}_1,...,{\underline l}_r)=
\begin{cases}
1   \quad  \text{if} \quad ({\underline k},{\underline l}_1,...,{\underline l}_r) =({\underline 0},{\underline 0},...,{\underline 0}), \cr
\frac{
1-\exp\big\{-i(N+1)\big(
\sum_{p \in \pp}h_1k_p\log p+\sum_{j=1}^{r}\sum_{m \in \nn_0}h_{2j}l_{mj}\log(m+\alpha_j)
\big)\big\}
}{
(N+1)
\big(1-\exp\big\{-i\big(\sum_{p \in \pp}h_1k_p\log p+\sum_{j=1}^{r}\sum_{m \in \nn_0}h_{2j}l_{mj}\log(m+\alpha_j)\big)\big\}
\big)} \cr
\qquad \text{if} \quad
({\underline k},{\underline l}_1,...,{\underline l}_r) \not =({\underline 0},{\underline 0},...,{\underline 0}).
\end{cases}
\end{align*}
Hence,
$$
\lim_{N  \to \infty}{\underline g}_N({\underline k},{\underline l}_1,...,{\underline l}_r)=
\begin{cases}
1 & \text{if } \quad ({\underline k},{\underline l}_1,...,{\underline l}_r)=({\underline 0},{\underline 0},...,{\underline 0}), \cr
0 & \text{otherwise}.
\end{cases}
$$

This, taking into account the continuity theorem for probability measures on compact groups (see \cite{HH-1977}), we obtain the assertion of the lemma.
\end{proof}

Now we continue the proof of Theorem \ref{RK-KM-new-joint-lim}.

Our next task is to pass from $\up_{N,n}$ to $\up_N$.   For this purpose we approximate in mean  $\uZ(s_1+ikh_1,s_{21}+ikh_{21},...,s_{2r}+ikh_{2r})$  by $\uZ_n(s_1+ikh_1,s_{21}+ikh_{21},...,s_{2r}+ikh_{2r})$ and $\uZ(s_1+ikh_1,s_{21}+ikh_{21},...,s_{2r}+ikh_{2r},\uom)$ by $\uZ_n(s_1+ikh_1,s_{21}+ikh_{21},...,s_{2r}+ikh_{2r},\uom)$, respectively,  as in Lemma~\ref{RK-le-2}. This can be done using known mean value results (see  \cite{RK-2000}, \cite{AL-RM-2009}).

Then we can show that both measures $\up_{N}$ and $\up_{N,\uom}$ converges weakly to a certain probability measure $\up$, similar to Lemma~\ref{RK-le-3}. Finally, using ergodic theory, we prove that $\up=\uuZ$. It is an analogue of Lemma~\ref{RK-le-4}, where instead of $a_{\alpha,h_1,h_2}$ and $\omega \in \Omega$ we use
$$a_{\uA,\uhh}:=\big((p^{ih_1}: p \in \pp),((m+\alpha_1)^{-ih_{21}}: m \in \nn_0),
\ldots,((m+\alpha_r)^{-ih_{2r}}: m \in \nn_0)\big)$$
and $\uom \in \uO$, respectively. This completes the proof of Theorem~\ref{RK-KM-new-joint-lim}.
\end{proof}

\begin{proof}[Proof of Theorem~\ref{RK-KM-new-joint-disc-univ}]

First we construct the support ${\underline S}_\uZ$ of the probability measure $\uuZ$.
Let $S_\varphi$ be the same as in Lemma~\ref{RK-le-5}. Arguing in a way similar
to the case of the measure $P_\uZ$, we can prove  that the support of the measure $\uuZ$ is the set ${\underline S}_\uZ:=S_\varphi \times H^{r}(D_T)$.

The rest of the proof is again standard, using Theorem \ref{RK-KM-new-joint-lim} and
Mergelyan's theorem, and we omit the details.

\end{proof}

%%%%%%
% References
%%%%%%

\end{document}